\documentclass{amsart}
\usepackage[colorlinks=true,citecolor=blue,urlcolor=magenta]{hyperref}
\usepackage[ansinew]{inputenc}
\usepackage{graphicx}
\usepackage{amsmath}
\usepackage{amsthm,enumitem}
\usepackage{amssymb, color}%
\usepackage[numbers, square]{natbib}
\usepackage{mathrsfs}
\usepackage{bbm}
\usepackage{xcolor}

\newcommand{\R}{\mathbb{R}}
\newcommand{\E}{\mathbb{E}}
\newcommand{\N}{\mathbb{N}}
\newcommand{\C}{\mathbb{C}}

\newcommand{\LCM}{{\rm lcm}}
\newcommand{\GCD}{{\rm gcd}}
\renewcommand{\P}{\mathbb{P}}
\newcommand{\1}{\mathbbm{1}}

\renewcommand {\leq}{\leqslant}
\renewcommand {\geq}{\geqslant}

\renewcommand{\Re}{\operatorname{Re}}

\newcommand{\todistr}{\overset{d}{\underset{n\to\infty}\longrightarrow}}

\newcommand{\toasm}{\overset{a.s.}{\underset{m\to\infty}\longrightarrow}}

\theoremstyle{plain}
\newtheorem{theorem}{Theorem}[section]
\newtheorem{lemma}[theorem]{Lemma}
\newtheorem{corollary}[theorem]{Corollary}
\newtheorem{proposition}[theorem]{Proposition}

\theoremstyle{definition}

\theoremstyle{remark}
\newtheorem{remark}[theorem]{Remark}

\begin{document}

\title[On the least common multiple of several random~integers]{On the least common multiple \\ of several random~integers}

\author{Alin Bostan}
\address{Alin Bostan, Inria, Universit\'e Paris-Saclay, 1 rue Honor\'e d'Estienne d'Orves, 91120 Palaiseau, France}
\email{alin.bostan@inria.fr}

\author{Alexander Marynych}
\address{Alexander Marynych, Faculty of Computer Science and Cybernetics, Taras Shev\-chen\-ko National University of Kyiv, 01601 Kyiv, Ukraine}
\email{marynych@unicyb.kiev.ua}

\author{Kilian Raschel}
\address{Kilian Raschel, CNRS \& Institut Denis Poisson, Universit\'{e} de Tours and Universit\'e d'Orl\'eans, 37200 Tours, France}\email{raschel@math.cnrs.fr}
\thanks{This project has received funding from the European Research Council (ERC) under the European Union's Horizon 2020 research and innovation programme under the Grant Agreement No 759702.}

\begin{abstract}
Let $L_n(k)$ denote the least common multiple of $k$ independent random
integers uniformly chosen in $\{1,2,\ldots ,n\}$. In this note, using a
purely probabilistic approach, we derive a criterion for the convergence in
distribution as $n\to\infty$ of $\frac{f(L_n(k))}{n^{rk}}$ for a wide class
of multiplicative arithmetic functions~$f$ with polynomial
growth $r>-1$. Furthermore, we identify the limit as an infinite
product of independent random variables indexed by prime numbers.
Along the way, we compute the generating function of a trimmed sum of
independent geometric laws, occurring in the above infinite product. This
generating function is rational; we relate it to the generating function of a
certain max-type Diophantine equation, of which we solve a generalized
version. Our results extend theorems by Erd\H{o}s and Wintner (1939),
Fern\'{a}ndez and Fern\'{a}ndez (2013) and Hilberdink and T\'{o}th
(2016).
\end{abstract}



\maketitle

\section{Introduction}



A celebrated result due to Dirichlet~\cite{Dirichlet:1849} states that
two random positive integers are coprime with probability~$6/\pi^2 \approx
0.61$. A heuristic argument goes as follows. A prime~$p$ divides a random
integer $X$ with probability $1/p$, and does not divide independent $X_1$ and $X_2$ simultaneously with probability $1 - 1/p^2$. Hence the event $\GCD(X_1, X_2) = 1$ occurs with probability
\[\prod_{p \in \mathcal{P}} \left(1 - \frac{1}{p^2} \right) = \left(
\sum_{n\geq 1} \frac{1}{n^2}\right)^{-1} = \frac{6}{\pi^2},\] 
where $\mathcal{P}=\{2, 3, 5, \ldots\}$ denotes the set of prime numbers. An
equivalent restatement is that two random positive integers admit an expected
number of $\pi^2/6\approx 1.64$ common positive integer divisors, or that the
expected number of integers between $1$ and~$N$ which are coprime with~$N$
equals $6N/\pi^2 \approx 0.61 N$. More generally, Ces\`aro
showed~\cite{Cesaro:1884,Cesaro:1885} that for $k\geq 2$ positive random
integers, the probability that they are relatively prime is $1/\zeta(k)$,
where $\zeta(s) = \sum_{n \geq 1} n^{-s}$ is the Riemann zeta function. For a
nice account of the rich history of Dirichlet's result, see~\cite{AbNi:2017}.


As stated, these facts are however not very precise, since there is no
uniform distribution on the set of positive integers. What we have implicitly
considered above is the uniform distribution on $\{1,2,\ldots, n\}$ and then
we have taken the limit as~$n$ goes to infinity. Formally, 
if $P_k(n)$ denotes the probability that $k\geq 2$ positive integers, chosen uniformly at random 
from $\{1,2,\ldots, n\}$, are relatively prime, i.e.
$$
P_k(n)=\frac{1}{n^k}\#\{(n_1,n_2,\ldots,n_k)\in\mathbb{N}^k:n_1,\ldots,n_k\leq n, \GCD(n_1,n_2,\ldots,n_k)=1\},
$$
then
$$
\lim_{n\to\infty}P_k(n) =1/\zeta(k).
$$
Moreover, the following estimates for the rate of convergence are known $P_k(n)=1/\zeta(k) + O(1/n)$ for $k\geq 3$, and $P_2(n)=1/\zeta(2) + O(\log n/n)$, see e.g.~\cite{Nymann:1972} and~\cite{DiEr:2004}. Further refinements of these celebrated results can be found in the recent papers \cite{Ferraguti+Micheli:2016,Mehrdad+Zhu:2016,Micheli+Schnyder:2016}.

Ces\`aro also considered similar questions when the greatest common
divisor (gcd) is replaced by the least common multiple (lcm). He proved
in~\cite{Cesaro:1885b} that the expected lcm of two random integers is asymptotically equal
to their product multiplied by the constant $\zeta(3)/\zeta(2)\approx 0.73$,
and more generally that if $X_1^{(n)}$ and $X_2^{(n)}$ are independent copies of a
random variable with the uniform distribution on $\{1,2,\ldots,n\}$, then the
moments $\E\{\LCM (X_1^{(n)},X_2^{(n)})^r\}$ of their least common multiple behave
like 
\begin{multline*}
\E\{\LCM (X_1^{(n)},X_2^{(n)})^r\}~\sim~\zeta(r+2)/\zeta(2) \cdot (\E (X_1^{(n)})^r)^2\\
~\sim~\frac{\zeta(r+2)}{\zeta(2)(r+1)^2} \cdot n^{2r},\quad n\to\infty.
\end{multline*} In contrast with the case of the gcd, the extension of this result
to the lcm of several random integers is much more subtle. This is the topic
of the current note.

Let thus $X_1^{(n)},X_2^{(n)},\ldots,X_k^{(n)}$ be independent copies
of a random variable $X^{(n)}$ with the uniform distribution on
$\{1,2,\ldots,n\}$. In what follows, we are interested in asymptotic
properties of the distribution of the least common multiple
\begin{equation*}
     L_n(k)=\LCM\left(X_1^{(n)},X_2^{(n)},\ldots,X_k^{(n)}\right),
\end{equation*}
as $n\to\infty$, and more generally of the quantity $f(L_n(k))$, for a wide class
of multiplicative arithmetic functions $f: \N \to \C$, with $\N$ denoting
$\{1,2,3,\ldots \}$. Recall that a function $f$ is said to be \emph{arithmetic} 
if its domain of definition is $\mathbb{N}$ and its range 
is $\mathbb{C}$. An arithmetic function is called
\emph{multiplicative} if $f(1)=1$ and if $f(mn)=f(m)f(n)$ as soon as
$m$ and $n$ are coprime. 

Our motivation for the present paper comes from two recent works, one by Fern\'{a}ndez and
Fern\'{a}ndez~\cite{Fernandez+Fernandez:2013} and the other by Hilberdink and
T\'{o}th~\cite{Hilberdink+Toth:2016}.

In 2013 Fern\'{a}ndez and Fern\'{a}ndez proved, see~Theorem~3(b) in~\cite{Fernandez+Fernandez:2013}, a generalization of Ces\`aro's
result for the lcm of three random integers. More precisely, they showed that
the moments $\E\{(L_n(3))^r\}$ behave asymptotically like $\frac{C_{r,3}}{(r+1)^3} \cdot n^{3r}$ as $n$ tends to infinity for every fixed $r\in\N$. Here, the
constant $C_{r,3}$ is equal (in the notation
of~\cite{Fernandez+Fernandez:2013}) to $C_{r,3} = T_3\zeta(2r+3)J(r+2)$, where
$T_3 = \prod_{p \in \mathcal{P}} (1-1/p)^2 (1+2/p)$ is the asymptotic
proportion of triples of integers that are pairwise coprime, and where
$J(r+2)$ is the Dirichlet series ${J(r+2) = \prod_{p \in \mathcal{P}} \left(1
+ \frac{3(p+1)}{(p+2)(p^{r+2}-1)}\right)}$. An easy computation shows that the
constant $C_{r,3}$ admits the equivalent expression
\begin{equation}\label{eq:c3r} 
     C_{r,3} =
\zeta(r+2)\zeta(2r+3)\prod_{p\in\mathcal{P}}\left(1-\frac{1}{p}\right)^2\left(1+\frac{2}{p}+\frac{2}{p^{r+2}}+\frac{1}{p^{r+3}}\right).
\end{equation}
In particular the expected lcm of three random positive integers is asymptotically equal to
their product multiplied by the constant $C_{1,3} \approx 0.34$. The method
used by Fern\'{a}ndez and Fern\'{a}ndez~\cite[\S4]{Fernandez+Fernandez:2013}
relies on probabilistic arguments combined with the classical identity
\begin{equation*}
     \LCM(X_1,X_2,X_3) = \frac{X_1X_2X_3 \GCD(X_1,X_2,X_3)}{\GCD(X_1, X_2) \GCD(X_2,X_3) \GCD(X_3,X_1)}.
\end{equation*}
Although this identity does admit a generalization for $k>3$ integers, the
probabilistic arguments used in \cite{Fernandez+Fernandez:2013} do not seem to extend smoothly to the case $k>3$.

Instead of that, for arbitrary $k\in\N$, Fern\'{a}ndez and Fern\'{a}ndez
provide in Theorem~1 in \cite{Fernandez+Fernandez:2013} upper (resp.\ lower) bounds
for the upper (resp.\ lower) limit of the probability $\P\{L_n(k)\leq x n^k\}$,
$x\in(0,1)$, but these upper and lower bounds are different. Only for $k=2$
and $k=3$ these bounds imply that the sequence $(\E
\{(L_n(k)/n^k)^{r}\})_{n\in\N}$ actually converges to a nondegenerate limit,
which is $(r+1)^{-2}\zeta(r+2)/\zeta(2)$ when $k=2$ and the aforementioned constant
$(r+1)^{-3}C_{r,3}$ when $k=3$.

It is natural to ask whether such a convergence result also holds for $k>3$.
The \emph{positive} answer to this question is implicit in the work of
Hilberdink and T\'{o}th~\cite{Hilberdink+Toth:2016}, see Theorem 2.1 therein.
Generalizing both the results of Ces\`aro~\cite{Cesaro:1885b} (for
$k=2$) and Fern\'{a}ndez and Fern\'{a}ndez~\cite{Fernandez+Fernandez:2013}
(for $k=3$), they managed to prove that for any $k\geq 2$ and $r\in\N$, the moments $\E \{(L_n(k))^r\}$ behave asymptotically like $(r+1)^{-k}C_{r,k} \cdot n^{kr}$ as $n$ tends to infinity, where the constant
$C_{r,k}$ is equal to 
\begin{equation}\label{eq:ckr}
	C_{r,k} = \prod_{p \in \mathcal{P}} \left( 1 -
\frac{1}{p}\right)^k \sum_{\ell_1, \ldots, \ell_k=0}^{\infty}
\frac{p^{r \, \max(\ell_1, \ldots, \ell_k)}}{p^{(r+1)(\ell_1 + \cdots + \ell_k)}}.
\end{equation} 
Hilberdink and T\'{o}th also proved, see Corollary~1 in \cite{Hilberdink+Toth:2016}, that
the $k$-variate sum above simplifies in the cases $k=2$, $k=3$ and $k=4$ to an
explicit rational function in $1/p$, allowing to retrieve the value $C_{r,2} =
\zeta(r+2)/\zeta(2)$ due to Ces\`aro, and the value $C_{r,3}$ in
Eq.~\eqref{eq:c3r} due to Fern\'{a}ndez and Fern\'{a}ndez.
The method used by Hilberdink and T\'{o}th for $k\in\{2,3,4\}$ is effective
and could yield an algorithm that computes (in principle) a formula similar to~\eqref{eq:c3r} for any
given~$k$. However, the algorithm has complexity exponential in~$k$,
so in practice it yields formulas for few values of~$k$.

One of the byproducts of the present work is that we further simplify the
expression of $C_{r,k}$ in~\eqref{eq:ckr}. Precisely, we prove, see
Corollary~\ref{prop:r_k_prop} below, that
\[ C_{r,k} =  \prod_{p\in\mathcal{P}} F_{r,k} \left( \frac{1}{p}\right),\quad k,r\in\N.\]
where $F_{r,k}(x)$ is the following \emph{explicit univariate} rational function:
\[
F_{r,k}(x) = 
\left(\frac{1-x}{1-{x^{r+1}}}\right)^{k} \cdot \sum_{j=1}^{k}\binom{k}{j}(-1)^{j-1}\frac{1-{x^{j(r+1)}}}{1-{x^{(j-1)(r+1)+1}}},\quad k,r\in\N,\quad |x|<1.
\]
The fact that the term in the product defining $C_{r,k}$
in~\eqref{eq:ckr} is a rational function in $1/p$ is not surprising. This
follows from the fact that if $\alpha_{r,k,\ell}$ denotes the number of
solutions in $\N_0^k$, where $\N_0 = \N \cup \{ 0 \}$, of the max-type {linear Diophantine} equation
\begin{equation}\label{eq:eq_integers}
     (r+1)(\ell_1+\cdots+\ell_k)-r \, \max(\ell_1, \ldots, \ell_k)=\ell,\quad k,r\in\N, \quad \ell\in\N_0,
\end{equation}
then a classical result due to Ehrhart~\cite{Ehrhart:1967} implies that the
generating functions
\begin{equation*}
     x\mapsto \sum_{\ell=0}^{\infty}\alpha_{r,k,\ell}x^{\ell}
\end{equation*}
are rational. Indeed, one can split the orthant $\N_0^k$
into wedges $x_{\sigma(1)}\geq \cdots \geq x_{\sigma(k)}$, where $\sigma$ is a
permutation of $\{1,\ldots,k\}$, get a rational generating function on each
wedge by~\cite{Ehrhart:1967}, and use inclusion-exclusion to take care of the
boundaries where the regions intersect. What is more interesting in our case
is that we get an \emph{explicit generating function}. Details are given in the Appendix.



A trivial consequence of Theorem 2.1 in \cite{Hilberdink+Toth:2016} is that for every $r\in\N$, the sequence of moments $(\E \{(L_n(k)/n^k)^{r}\})_{n\in\N}$ converges to the constant
$(r+1)^{-k}C_{r,k}$ as $n\to\infty$, whence, by the classical method of moments (see
e.g.\ Example (d) on page 251 in \cite{Feller:1971})
the following convergence in distribution holds
\begin{equation}\label{eq:conv_in_distr_l_n}
     \frac{L_n(k)}{n^k}\todistr Y_{\infty,k},
\end{equation}
where $Y_{\infty,k}$ is a random variable with values in $[0,1]$ such that $\E Y^{r}_{\infty,k}=(r+1)^{-k}C_{r,k}$ for all $r\in\N$. Conversely, since the sequence $(\frac{L_n(k)}{n^k})_{n\in\mathbb{N}}$ is uniformly bounded by~$1$, the convergence in distribution \eqref{eq:conv_in_distr_l_n} yields the convergence of the moments, and thereby a particular case of Theorem~2.1 in \cite{Hilberdink+Toth:2016} when restricted to power functions $f(n)=n^r$. The aforementioned Theorem~2.1 in \cite{Hilberdink+Toth:2016} provides general conditions on a multiplicative function $f$ of a polynomial growth $r>-1$ that ensure the convergence of moments 
\begin{equation}
\label{eq:moments}
     \E\left\{\frac{f(L_n(k))}{n^{rk}}\right\},
\end{equation}
as $n\to\infty$, to a finite positive limit. The approach used in \cite{Hilberdink+Toth:2016} to derive convergence of \eqref{eq:moments} is purely analytical. Even in the simple case  \eqref{eq:conv_in_distr_l_n} it  does not shed light on the probabilistic mechanisms behind this convergence, nor on the probabilistic structure of the limit~$Y_{\infty,k}$. Moreover, in general it does not provide a distributional convergence of
\begin{equation}
\label{eq:random_variable}
     \frac{f(L_n(k))}{n^{rk}},\quad k\in\N,
\end{equation}
as $n\to\infty$. The main contributions of the current note is a derivation of a criterion for the convergence in distribution of \eqref{eq:random_variable}, as $n\to\infty$, by using a purely probabilistic approach, see Theorem \ref{thm:main1} below. Furthermore, we manage to identify the limit of \eqref{eq:random_variable} as an infinite product of independent random variables indexed by the set of prime numbers $\mathcal P$. Further comparison of our main results and Theorem~2.1 in \cite{Hilberdink+Toth:2016} shall be given in Remark \ref{rem:comaprison} below.

As we shall see, our main result is very close in spirit to a well-known
result in probabilistic number theory, namely the celebrated
Erd\H{o}s--Wintner theorem, see for example \cite{ErWi:1939} or Theorem 3 in 
\cite{Galambos:1970}. Let us recall that the latter asserts that
if $X^{(n)}$ is a random variable with uniform distribution on
$\{1,2,\ldots,n\}$ and if~$f$ is an \emph{additive} arithmetic function, then
the sequence $(f(X^{(n)}))_{n\in\N}$ converges in distribution if and only
if the following three series converge for some $A>0$:
\begin{equation}
\label{eq:erdos-wintner_cond}
     \sum_{p\in\mathcal P,\,\vert f(p)\vert>A}\frac{1}{p},\qquad 
     \sum_{p\in\mathcal P,\,\vert f(p)\vert\leq A}\frac{f(p)}{p},\qquad 
     \sum_{p\in\mathcal P,\,\vert f(p)\vert\leq A}\frac{f^2(p)}{p}.
\end{equation}
Moreover, if the limit $X_{\infty}$ of $(f(X^{(n)}))_{n\in\N}$ exists, it 
necessarily satisfies
\begin{equation*}
     \E e^{it X_{\infty}}=\prod_{p\in\mathcal P}\left(1-\frac{1}{p}\right)\sum_{j=0}^{\infty}\frac{e^{it f(p^j)}}{p^{j}},\quad t\in\R,
\end{equation*}
and thus $X_{\infty}$ is a sum of independent random variables indexed by primes. The
underlying probabilistic result behind the Erd\H{o}s--Wintner result is
Kolmogorov's three series theorem, see Chap.~III.4 in \cite{Tenenbaum:1995}. 
Let us further point out that by Kolmogorov's three series theorem, the conditions~\eqref{eq:erdos-wintner_cond} are equivalent to the almost sure convergence of the series 
\begin{equation}
\label{eq:rep_X_infinity}
     \sum_{p\in\mathcal{P}}f\Bigl(p^{\mathcal{G}(p)}\Bigr),
\end{equation}
where $(\mathcal{G}(p))_{p\in\mathcal{P}}$ is a family of mutually independent geometric random variables, such that
\begin{equation}\label{eq:geometrics}
     \P\{\mathcal{G}(p)=k\}=\left(1-\frac{1}{p}\right)\frac{1}{p^k},\quad k\in\N_0=\{0,1,2,\ldots\}.
\end{equation}
Thus \eqref{eq:rep_X_infinity} is a representation of $f(X_{\infty})$, the limit of $(f(X^{(n)}))_{n\in\N}$ as $n\to\infty$. 

Let us finally mention some recent works related to the problem considered here. In two recent papers~\cite{AKM:2018,CillRueSarka:14} the authors analyze an asymptotic behavior of $\LCM(A_n)$, where $A_n$ is a random subset of $\{1,2,\ldots,n\}$ obtained by removing every element with a fixed probability $p\in(0,1)$. Since in this case the cardinality of $A_n$ increases linearly as $n\to\infty$, the model exhibits a completely different asymptotic behavior, see e.g.~Corollary~1.5 in \cite{AKM:2018}. Another related problem was 
addressed in \cite{Hu:2013,Toth:2016}, where it was proved that the set of $k$-tuples of positive integers such that any $m$ of them are relatively prime possesses an asymptotic density. Similarly to our results the explicit formula for this density involves product over $p\in\mathcal{P}$ of rational functions of $1/p$, see Eq.~(6) in \cite{Toth:2016}.

We close the introduction by setting up some notation. We shall
denote by $\lambda_p(n)$ the exponent of the prime number
$p\in\mathcal{P}$ in the prime factorization of $n\in\N$, that is
\begin{equation*}
     n=\prod_{p\in\mathcal{P}}p^{\lambda_p(n)}.
\end{equation*}
Note that $\lambda_p(n)$ is zero for all but finitely many $p\in\mathcal{P}$. We shall further ubiquitously use the family $(\mathcal{G}_j(p))_{j\in \{1,\ldots,k\},p\in\mathcal{P}}$ of mutually independent random variables such that $\mathcal{G}_j(p)$ is distributed like $\mathcal{G}(p)$ in \eqref{eq:geometrics} for every $j=1,2,\ldots,k$. 
Finally, given any $i\in\mathbb N$, we shall denote by $\vee_{j=1}^{i}a_k$ the maximum of real numbers $a_1,\ldots ,a_i$.

\section{Main results}
\label{sec:main_result}

Given a multiplicative function $f$ and $r\in\R$, define the infinite random product
\begin{equation}\label{eq:limit_definition}
     X_{f,\infty,k}:=\prod_{p\in\mathcal{P}}\frac{f(p^{\vee_{j=1}^{k}\mathcal{G}_j(p)})}{p^{r\sum_{j=1}^{k}\mathcal{G}_j(p)}}.
\end{equation}
We characterize the convergence of $X_{f,\infty,k}$ in Proposition~\ref{prop:limit_convergence} below. The denominators in the infinite product \eqref{eq:limit_definition} should be thought of as normalization factors. Note also that taking $f$ as the identity function and $r=1$, the quantity $X_{f,\infty,k}$ becomes
\begin{equation}\label{eq:r_k_def}
     R_k:=\prod_{p\in\mathcal{P}}p^{\vee_{j=1}^{k}\mathcal{G}_j(p)-\sum_{j=1}^{k}\mathcal{G}_j(p)}\in1/\mathbb N.
\end{equation}
The ordinary generating function of $R_k^{-1}$ and the moments of $R_k$ will be computed in Proposition~\ref{prop:z_k_p_prop} and Proposition~\ref{prop:r_k_prop}.

For $p\in\mathcal{P}$ and $r\in\R$, put 
\begin{equation}\label{eq:bpr_definition}
     B_{p,r}=\log \left(\frac{\vert f(p)\vert}{p^r}\right).
\end{equation}

\begin{proposition}\label{prop:limit_convergence}
The infinite product on the right-hand side of \eqref{eq:limit_definition} converges a.s.\ if and only if the following three assumptions are satisfied: for some $A>0$,
\begin{enumerate}[label={(\alph{*})},ref={(\alph{*})}]
     \item\label{it:1}the series $\sum_{p\in\mathcal{P}}\frac{\1_{\{\vert B_{p,r}\vert\geq A\}}}{p}$ converges;
     \item\label{it:2}the series $\sum_{p\in\mathcal{P}}\frac{B_{p,r}\1_{\{\vert B_{p,r}\vert\leq A\}}}{p}$ converges;
     \item\label{it:3}the series $\sum_{p\in\mathcal{P}}\frac{B^2_{p,r}\1_{\{\vert B_{p,r}\vert\leq A\}}}{p}$ converges.
\end{enumerate}
If moreover $\vert f(p)\vert \sim p^r$ as $p\to\infty$ along the prime numbers, then \ref{it:1} holds automatically, \ref{it:2} implies \ref{it:3}, and \ref{it:2} is equivalent to
\begin{enumerate}[label={(\alph{*})},ref={(\alph{*})}]
\setcounter{enumi}{3}
    \item\label{it:4}the series $\sum_{p\in\mathcal{P}}\frac{1}{p}\log\left(\frac{\vert f(p)\vert}{p^r}\right)$ converges.
\end{enumerate}
\end{proposition}
\begin{remark}\label{rem:comaprison2}
Note that the assumption (i) in \cite{Hilberdink+Toth:2016} implies $\vert f(p)\vert \sim p^r$ as $p\to\infty$ along the prime numbers, as well as \ref{it:4}. 
\end{remark}
In order to illustrate how demanding item \ref{it:4} above is, let us recall the most classical result on the Bertrand-type series: 
$$
\sum_{p\in\mathcal{P}}\frac{1}{p\log^{1-\varepsilon}p}<\infty\quad\Longleftrightarrow \quad \varepsilon\leq 0.
$$

The proof of Proposition~\ref{prop:limit_convergence}, as well as proofs of
all results from this section, are postponed to Section~\ref{sec:proofs}. With Proposition~\ref{prop:limit_convergence} at hand, we can formulate our main result.

\begin{theorem}\label{thm:main1}
Assume that $f$ is a multiplicative arithmetic function and that $r\in\R$.
The following statements are equivalent:
\begin{enumerate}[label={(\roman{*})},ref={(\roman{*})}]
     \item\label{it-main:1}the infinite product \eqref{eq:limit_definition} defining $X_{f,\infty,k}$ converges~a.s.;
     \item\label{it-main:2}the conditions \ref{it:1}, \ref{it:2} and \ref{it:3} of Proposition~\ref{prop:limit_convergence} are satisfied;
     \item\label{it-main:3}$X_{f,\infty,k}$ converges~a.s.~and the following convergence in distribution holds
\begin{equation}\label{eq:thm1_conv1}
     \frac{f(L_n(k))}{\bigl(X_1^{(n)}X_2^{(n)}\cdots X_k^{(n)}\bigr)^{r}}\todistr X_{f,\infty,k};
\end{equation}
     \item\label{it-main:4}$X_{f,\infty,k}$ converges~a.s.~and
\begin{equation}\label{eq:thm1_conv2}
     \frac{f(L_n(k))}{n^{rk}}\todistr X_{f,\infty,k}\prod_{j=1}^{k}U^r_j,
\end{equation}
where $(U_j)_{j=1,\ldots,k}$ are independent copies of a random variable $U$ with the uniform distribution on $[0,1]$, and $(U_j)_{j=1,\ldots,k}$ are also independent of $X_{f,\infty,k}$.
\end{enumerate}
\end{theorem}

\begin{remark}
The identity function $f(n)=n$ obviously satisfies assumptions \ref{it:1}, \ref{it:2} and \ref{it:3} with $r=1$, thus with our notation \eqref{eq:r_k_def},
\begin{equation}\label{eq:thm1_conv3}
     \frac{L_n(k)}{X_1^{(n)}X_2^{(n)}\cdots X_k^{(n)}}\todistr R_k,
\end{equation}
and
\begin{equation}\label{eq:thm1_conv4}
     \frac{L_n(k)}{n^k}\todistr R_k\prod_{j=1}^{k}U_j=Y_{\infty,k}.
\end{equation}
The quantity $R_k$ in \eqref{eq:r_k_def} is a.s.\ positive. As we have already mentioned in the introduction, both \eqref{eq:thm1_conv3} and \eqref{eq:thm1_conv4} follow from the results of \cite{Hilberdink+Toth:2016}. 
\end{remark}

\begin{remark}\label{rem:comaprison}
Let us now compare our Theorem \ref{thm:main1} with Theorem 2.1 in \cite{Hilberdink+Toth:2016} in more details. Whereas Hilberdink and T\'{o}th's main focus is placed on the convergence of the first moments \eqref{eq:moments} of the variables \eqref{eq:random_variable} (and actually of \emph{all} moments, because in \eqref{eq:moments} one may replace $f(L_n(k))/n^{rk}$ by $f(L_n(k))^q/n^{rkq}$), our Theorem \ref{thm:main1} provides much less restrictive conditions, see Remark \ref{rem:comaprison2} above, ensuring the convergence in distribution of \eqref{eq:random_variable}. Obviously, these results do overlap in some particular cases: convergence of moments can give convergence in distribution (e.g.\ if the method of moments applies, as for~\eqref{eq:conv_in_distr_l_n}); conversely, convergence in distribution may yield convergence of moments (for example, when the limit is compactly supported). But in general, they are of different nature. Furthermore, the limiting random variable~\eqref{eq:limit_definition}, being almost surely finite under assumptions \ref{it:1}, \ref{it:2} and \ref{it:3} in Proposition~\ref{prop:limit_convergence}, might have infinite power moments. Thereby in general we cannot expect convergence of the moments under \ref{it:1}, \ref{it:2} and \ref{it:3} alone. Another important observation is that we do not need any assumptions about the behavior of $f(p^q)$ for $q>1$ (condition (ii) in \cite{Hilberdink+Toth:2016}). Indeed, as we shall show in Section~\ref{sec:proofs}, powers of primes do not have impact in the~a.s.~convergence of the infinite product which defines $X_{f,\infty,k}$. Note that the same phenomenon occurs in the Erd\H{o}s--Wintner theorem, see conditions \eqref{eq:erdos-wintner_cond}. On the other hand, the behavior of $f(p^q)$ should impact the finiteness of power moments of $X_{f,\infty,k}$ explaining the appearance of condition (ii) in \cite{Hilberdink+Toth:2016}.
\end{remark}

Let us close Section~\ref{sec:main_result} by studying some properties of the random variable $R_k$ in~\eqref{eq:r_k_def}. Plainly, it is an infinite product of blocks along primes $p\in\mathcal P$, each of them being equal to $1/p$ raised to the power~$Z_k(p)$, where
\begin{equation}
\label{eq:sup-sum}
     Z_k(p):=\sum_{j=1}^{k}\mathcal{G}_j(p)-\vee_{j=1}^{k}\mathcal{G}_j(p).
\end{equation}
Besides the very particular case $k=2$, for which the latter reduces to $\mathcal{G}_1(p)\wedge\mathcal{G}_2(p)$ (and thus everything is known), the law of the random variable in \eqref{eq:sup-sum} is not trivial. Let us mention in passing that quantities
\begin{equation*}
     {}^{(r)}S_n=\xi_1+\xi_2+\cdots+\xi_n-\xi_n^{(1)}-\cdots-\xi_n^{(r)},
\end{equation*}
where $(\xi_k)_{k\in\N}$ are iid random variables and $\xi_n^{(n)}\leq\cdots\leq \xi_n^{(2)}\leq \xi_n^{(1)}$ is their arrangement in nondecreasing order, are called {\it trimmed sums}, see for instance \cite{CsSi:1995}. However, we have not been able to locate in the vast body of literature on trimmed sums any results about the exact distribution of $Z_k(p)$.

\begin{proposition}
\label{prop:z_k_p_prop}
Let $k\in\N$ and  $p\in\mathcal{P}$. The ordinary generating function of $Z_k(p)$ is rational and is given for $\vert t\vert \leq p$ by
\begin{equation*}
     \E t^{Z_k(p)}=\left(1-\frac{1}{p}\right)^{k}\left(1-\frac{t}{p}\right)^{-k}\sum_{j=1}^{k}\binom{k}{j}(-1)^{j-1}
\frac{1-\left(\frac{t}{p}\right)^j}{1-\frac{t^{j-1}}{p^j}}. 
\end{equation*}
In particular, one has $\E t^{Z_1(p)}=1$, as well as
\begin{align*}
     \E t^{Z_2(p)}&=\E t^{\mathcal{G}_1(p)\wedge\mathcal{G}_2(p)}=\frac{1-\frac{1}{p^2}}{1-\frac{t}{p^2}},\\
     \E t^{Z_3(p)}&=\frac{\left(1-\frac{1}{p}\right)^2}{\left(1-\frac{t}{p^2}\right)\left(1-\frac{t^2}{p^3}\right)} \left(1+\frac{2}{p}+\frac{2t}{p^2}+\frac{t}{p^3}\right).
\end{align*}
\end{proposition}

Notice that the expression of $\E t^{Z_2(p)}$ above is clear, as $\mathcal{G}_1(p)\wedge\mathcal{G}_2(p)$ is distributed as $\mathcal{G}(p^2)$.

Using Proposition~\ref{prop:z_k_p_prop} we immediately obtain the following corollary generalizing formulas (11) and (12) in \cite{Hilberdink+Toth:2016}.

\begin{corollary}\label{prop:r_k_prop}
For $r\in\N_0$ we have
\begin{multline}\label{eq:moments_of_r_k}
     \E R_k^{r}= \E \prod_{p\in\mathcal{P}}p^{-rZ_k(p)}
     = \prod_{p\in\mathcal{P}}\frac{\left(1-\frac{1}{p}\right)^{k}}{\left(1-\frac{1}{p^{r+1}}\right)^{k}}\sum_{j=1}^{k}\binom{k}{j}(-1)^{j-1}\frac{1-\frac{1}{p^{j(r+1)}}}{1-\frac{1}{p^{(j-1)(r+1)+1}}}.
\end{multline}
In particular, using the Euler product of the Riemann zeta-function
\begin{equation*}
     \zeta(s)=\prod_{p\in\mathcal{P}}\left(1-\frac{1}{p^s}\right)^{-1},\quad \Re s>1,
\end{equation*}
we obtain $\E R_1^{r}=1$, as well as
\begin{align*}
     C_{r,2}=\E R_2^{r}&=\frac{\zeta(r+2)}{\zeta(2)},\\
     C_{r,3}=\E  R_3^{r}&=\zeta(r+2)\zeta(2r+3)\prod_{p\in\mathcal{P}}\left(1-\frac{1}{p}\right)^2\left(1+\frac{2}{p}+\frac{2}{p^{r+2}}+\frac{1}{p^{r+3}}\right).
\end{align*}
\end{corollary}
In general, it is not possible to further simplify the above Euler products.
Notice that many well-known constants (in particular in number theory) have
Euler product expansions, see e.g.~\cite{Moree00, Niklasch02} and the
whole Chap.~2 in~\cite{Finch:2009}.

\section{Proofs}
\label{sec:proofs}

\begin{proof}[Proof of Proposition~\ref{prop:limit_convergence}]
Passing to logarithms, we see that the a.s.\ convergence of the infinite product is equivalent to the a.s.\ convergence of the series
\begin{equation*}
     \sum_{p\in\mathcal{P}}\left(\log \vert f(p^{\vee_{j=1}^{k}\mathcal{G}_j(p)})\vert-r\sum_{j=1}^{k}\mathcal{G}_j(p)\log p\right).
\end{equation*}
First of all, note that since $f(1)=1$ it is enough to show that
\begin{equation}\label{eq:proof_series_conv}
     \sum_{p\in\mathcal{P}}\left(\log \vert f(p^{\vee_{j=1}^{k}\mathcal{G}_j(p)})\vert-r\sum_{j=1}^{k}\mathcal{G}_j(p)\log p\right)\1_{\{\sum_{j=1}^{k}\mathcal{G}_j(p)\geq 1\}}
\end{equation}
converges a.s. Further, we apply the Borel--Cantelli lemma to check that for any $k\geq1$,
\begin{equation}\label{eq:proof_bc_arg}
     \P\left\{\sum_{j=1}^{k}\mathcal{G}_j(p) \geq  2 \text{ for infinitely many }p\in\mathcal{P}\right\}=0.
\end{equation}
Indeed,
\begin{align*}
\sum_{p\in\mathcal{P}}\P\left\{\sum_{j=1}^{k}\mathcal{G}_j(p) \geq  2\right\}&=\sum_{p\in\mathcal{P}}\left(1-\P\left\{\sum_{j=1}^{k}\mathcal{G}_j(p)=0\right\}-\P\left\{\sum_{j=1}^{k}\mathcal{G}_j(p)=1\right\}\right)\\
&=\sum_{p\in\mathcal{P}}\left(1-\left(1-\frac{1}{p}\right)^k-k\left(1-\frac{1}{p}\right)^{k}\frac{1}{p}\right)\\
&=\sum_{p\in\mathcal{P}}\left(1-\left(1-\frac{1}{p}\right)^k\left(1+\frac{k}{p}\right)\right)\\
&\leq\sum_{p\in\mathcal{P}}\left(1-\left(1-\frac{k}{p}\right)\left(1+\frac{k}{p}\right)\right)=k^2\sum_{p\in\mathcal{P}}\frac{1}{p^2}<\infty.
\end{align*}
Thus, the event $\{\sum_{j=1}^{k}\mathcal{G}_j(p) \geq  2\}$ occurs only for finitely many $p\in\mathcal{P}$ a.s.\ and the convergence of \eqref{eq:proof_series_conv} is equivalent to that of
\begin{equation}\label{eq:proof_series_conv2}
     \sum_{p\in\mathcal{P}}\left(\log \vert f(p)\vert-r\log p\right)\1_{\{\sum_{j=1}^{k}\mathcal{G}_j(p)=1\}}=\sum_{p\in\mathcal{P}}B_{p,r}\1_{\{\sum_{j=1}^{k}\mathcal{G}_j(p)=1\}},
\end{equation}
because obviously the event $\{\sum_{j=1}^{k}\mathcal{G}_j(p)=1\}$ implies $\{\vee_{j=1}^{k}\mathcal{G}_j(p)=1\}$. Note that the series in \eqref{eq:proof_series_conv2} consists of independent summands. Therefore, the assumptions \ref{it-main:1}, \ref{it-main:2} and \ref{it-main:3} are necessary and sufficient for the a.s.\ convergence of \eqref{eq:proof_series_conv2} by Kolmogorov's three series theorem (see page 317 in \cite{Feller:1971}), since
\begin{equation*}
     \P\left\{\sum_{j=1}^{k}\mathcal{G}_j(p)=1\right\}=k\left(1-\frac{1}{p}\right)^{k}\frac{1}{p}\sim \frac{k}{p},\quad p\to\infty.\qedhere
\end{equation*}
\end{proof}

The main ingredient in the subsequent proofs is contained in the following elementary lemma. Its first part is well known in the probabilistic literature and is given explicitly in \cite{Arratia+Barbour+Tavare:2003}, see formula (1.45) on page 28 therein. The second and third parts are just slight extensions thereof. Recall that $X^{(n)}$ denotes a random variable with uniform distribution on $\{1,2,\ldots,n\}$.

\begin{lemma}\label{lem:conv_to_geom}
Let
\begin{equation*}
     X^{(n)}=\prod_{p\in\mathcal{P}}p^{\lambda_p(X^{(n)})}
\end{equation*}
be the decomposition of $X^{(n)}\in \{1,2,\ldots,n\}$ into prime factors. Then
\begin{enumerate}[label={(\roman{*})},ref={(\roman{*})}]
     \item\label{itt:1}we have 
\begin{equation*}
     \bigl(\lambda_p(X^{(n)})\bigr)_{p\in\mathcal{P}}\todistr \left(\mathcal{G}(p)\right)_{p\in\mathcal{P}};
\end{equation*}
     \item\label{itt:2}we have
\begin{equation*}
     \left(\frac{X^{(n)}}{n},\bigl(\lambda_p(X^{(n)})\bigr)_{p\in\mathcal{P}}\right)\todistr \left(U,\left(\mathcal{G}(p)\right)_{p\in\mathcal{P}}\right),
\end{equation*} 
with $U$ being uniformly distributed on $[0,1]$ and independent of $\left(\mathcal{G}(p)\right)_{p\in\mathcal{P}}$;
     \item\label{itt:3}for $p,q\in\mathcal{P}$, $p\neq q$ and $k_p,k_q\in\N_0$, we have
\begin{equation*}
     \P\{\lambda_p(X^{(n)})=k_p,\lambda_q(X^{(n)})=k_q\}=\left(1-\frac{1}{p}\right)\left(1-\frac{1}{q}\right)\frac{1}{p^{k_p}q^{k_q}}+O\left(\frac{1}{n}\right),
\end{equation*}
where the constant in the $O$-term does not depend on $(p,q,k_p,k_q)$.
\end{enumerate}
\end{lemma}
\begin{proof}
Let us prove part \ref{itt:2}. Fix $x\in (0,1]$, $p\in\mathcal{P}$ and $c_2,c_3,\ldots,c_p\in\N_0$. One has
\begin{align*}
\P\{X^{(n)}\leq nx,&\,\lambda_2(X^{(n)})\geq c_2,\lambda_3(X^{(n)})\geq c_3,\ldots,\lambda_p(X^{(n)})\geq c_p\}\\
&=\P\{X^{(n)}\leq \lfloor nx\rfloor,X^{(n)}\text{ is divisible by } 2^{c_2}3^{c_3}\cdots p^{c_p}\}\\
&=\frac{1}{n}\#\{k\in\{1,2,\ldots,\lfloor nx\rfloor\}:k\text{ is divisible by }2^{c_2}3^{c_3}\cdots p^{c_p}\}\\
&=\frac{1}{n}\left\lfloor\frac{\lfloor nx\rfloor}{2^{c_2}3^{c_3}\cdots p^{c_p}}\right\rfloor\\&\underset{n\to\infty}{\longrightarrow}\frac{x}{2^{c_2}3^{c_3}\cdots p^{c_p}}
=\P\{U\leq x,\mathcal{G}(2)\geq c_2,\mathcal{G}(3)\geq c_3,\ldots,\mathcal{G}(p)\geq c_p\}.
\end{align*}
Part \ref{itt:1} obviously follows from part \ref{itt:2}. For the item \ref{itt:3}, notice that
\begin{equation*}
     \P\{\lambda_p(X^{(n)})\geq i,\lambda_q(X^{(n)})\geq j\}=\frac{1}{n}\left\lfloor\frac{n}{p^iq^j}\right\rfloor\in \Big(\frac{1}{p^iq^j}-\frac{1}{n},\frac{1}{p^iq^j}\Big]
\end{equation*}
and thus
\begin{multline*}
\P\{\lambda_p(X^{(n)})=k_p,\lambda_q(X^{(n)})=k_q\}\\
\in \left[\left(1-\frac{1}{p}\right)\left(1-\frac{1}{q}\right)\frac{1}{p^{k_p}q^{k_q}}-\frac{2}{n},\left(1-\frac{1}{p}\right)\left(1-\frac{1}{q}\right)\frac{1}{p^{k_p}q^{k_q}}+\frac{2}{n}\right].
\end{multline*}
The proof is complete.
\end{proof}

\begin{proof}[Proof of Theorem~\ref{thm:main1}]
With Lemma~\ref{lem:conv_to_geom} at hand, the proof of Theorem~\ref{thm:main1} is more or less straightforward. From Proposition~\ref{prop:limit_convergence}, we already know that \ref{it-main:1} and \ref{it-main:2} are equivalent.

Let us show that \ref{it-main:1} implies \ref{it-main:3}. Let us first write the prime power decompositions
\begin{equation*}
     X^{(n)}_j=\prod_{p\in\mathcal{P}}p^{\lambda_p(X^{(n)}_j)},\quad j=1,\ldots,k.
\end{equation*}
Then
\begin{equation*}
     L_n(k)=\LCM\{X^{(n)}_1,\ldots,X^{(n)}_k\}=\prod_{p\in\mathcal{P}}p^{\vee_{j=1}^{k}\lambda_p(X^{(n)}_j)},
\end{equation*}
and, using multiplicativity of $f$,
\begin{equation*}
     \frac{f(L_n(k))}{\bigl(X_1^{(n)}X_2^{(n)}\cdots X_k^{(n)}\bigr)^{r}}=\prod_{p\in\mathcal{P}}\frac{f(p^{\vee_{j=1}^{k}\lambda_p(X^{(n)}_j)})}{p^{r\sum_{j=1}^{k}\lambda_p(X^{(n)}_j)}}.
\end{equation*}
Fix $m\in\N$ and decompose
\begin{equation*}
     \frac{f(L_n(k))}{\bigl(X_1^{(n)}X_2^{(n)}\cdots X_k^{(n)}\bigr)^{r}}=\left(\prod_{p\in\mathcal{P},\,p\leq m}\cdots \right)\left(\prod_{p\in\mathcal{P},\,p > m}\cdots\right)=:Y_m(n)Z_m(n).
\end{equation*}
By Lemma~\ref{lem:conv_to_geom} \ref{itt:1} and the continuous mapping theorem, see Theorem 2.7 in \cite{Billingsley:1999},
\begin{equation*}
     Y_m(n)\todistr X_{f,\infty,k}(m):=\prod_{p\in\mathcal{P},\,p\leq m}\frac{f(p^{\vee_{j=1}^{k}\mathcal{G}_j(p)})}{p^{r\sum_{j=1}^{k}\mathcal{G}_j(p)}}.
\end{equation*}
By \ref{it-main:1} we have
\begin{equation*}
     X_{f,\infty,k}(m)\toasm X_{f,\infty,k}.
\end{equation*}
Denoting by $E_\varepsilon$ the event $\{\vert\log (\vert Z_m(n)\vert)\vert>\varepsilon\}$ and using Theorem 3.2 in \cite{Billingsley:1999}, it remains to show that for every fixed $\varepsilon>0$,
\begin{equation}\label{thm:proof_bill2}
     \lim_{m\to\infty}\limsup_{n\to\infty}\P\{E_\varepsilon\}=0.
\end{equation}
We have
\begin{align*}
\P\{E_\varepsilon\}& =  \P\left\{\text{for some }p\in\mathcal{P},\,p>m,\sum_{j=1}^{k}\lambda_p(X_j^{(n)})\geq 2, {E_\varepsilon}\right\}\\
&\quad +\P\left\{\text{for all }p\in\mathcal{P},\,p>m,\sum_{j=1}^{k}\lambda_p(X_j^{(n)})\leq 1, {E_\varepsilon}\right\}\\
&\leq\sum_{p\in\mathcal{P},\,p>m}\P\left\{\sum_{j=1}^{k}\lambda_p(X_j^{(n)})\geq 2\right\}\\
&\quad+\P\left\{\left\vert\sum_{p\in\mathcal{P},\,p>m}B_{p,r}\1_{\{\sum_{j=1}^{k}\lambda_p(X_j^{(n)})=1\}}\right\vert>\varepsilon\right\}=:P^{(1)}_m(n)+P^{(2)}_m(n).
\end{align*}
We deal with the latter two summands separately. For $P^{(1)}_m(n)$, we have
\begin{align*}
P^{(1)}_m(n)&\leq \sum_{p\in\mathcal{P},\,p>m}\P\{\lambda_p(X_j^{(n)})\geq 2\text{ for some }j=1,\ldots,k\}\\
&\quad+\sum_{p\in\mathcal{P},\,p>m}\P\{\lambda_p(X_i^{(n)})\geq 1,\lambda_p(X_j^{(n)})\geq 1\text{ for some }1\leq i,j\leq k,i\neq j\}\\
&\leq k\sum_{p\in\mathcal{P},\,p>m}\P\{\lambda_p(X^{(n)})\geq 2\}+k(k-1)\sum_{p\in\mathcal{P},\,p>m}(\P\{\lambda_p(X^{(n)})\geq 1\})^2\\
&=k\sum_{p\in\mathcal{P},\,p>m}\frac{1}{n}\left\lfloor\frac{n}{p^2}\right\rfloor+k(k-1)\sum_{p\in\mathcal{P},\,p>m}\left(\frac{1}{n}\left\lfloor\frac{n}{p}\right\rfloor\right)^2\\
&\leq k^2 \sum_{p\in\mathcal{P},\,p>m}\frac{1}{p^2}\to 0,\quad m\to\infty.
\end{align*}

To deal with $P^{(2)}_m(n)$ we pick $A>0$ such that the conditions \ref{it:1}, \ref{it:2} and \ref{it:3} in Proposition~\ref{prop:limit_convergence} hold. We have
\begin{align}
     P^{(2)}_m(n)&\leq \P\left\{\left\vert\sum_{p\in\mathcal{P},\,p>m}B_{p,r}\1_{\{\vert B_{p,r}\vert\geq A,\sum_{j=1}^{k}\lambda_p(X_j^{(n)})=1\}}\right\vert>\frac{\varepsilon}{2}\right\}\nonumber\\\nonumber
&\quad+
\P\left\{\left\vert\sum_{p\in\mathcal{P},\,p>m}B_{p,r}\1_{\{\vert B_{p,r}\vert \leq  A,\sum_{j=1}^{k}\lambda_p(X_j^{(n)})=1\}}\right\vert>\frac{\varepsilon}{2}\right\}\\
&\leq \P\left\{\text{for some }p\in\mathcal{P},\,p>m,\vert B_{p,r}\vert\geq A,\sum_{j=1}^{k}\lambda_p(X_j^{(n)})=1\right\}\nonumber\\
&\quad+
\P\left\{\left\vert\sum_{p\in\mathcal{P},\,p>m}B_{p,r}\1_{\{\vert B_{p,r}\vert \leq A,\sum_{j=1}^{k}\lambda_p(X_j^{(n)})=1\}}\right\vert>\frac{\varepsilon}{2}\right\}.\label{eq:second_probab_to_estimate}
\end{align}
The first probability can be estimated as follows
\begin{align*}
&\hspace{-2.5cm}\P\left\{\text{for some }p\in\mathcal{P},\,p>m,\vert B_{p,r}\vert\geq A,\sum_{j=1}^{k}\lambda_p(X_j^{(n)})=1\right\}\\
&\leq \sum_{p\in\mathcal{P},\,p>m}\1_{\{\vert B_{p,r}\vert\geq A\}}\P\left\{\sum_{j=1}^{k}\lambda_p(X_j^{(n)})=1\right\}\\
&\leq \sum_{p\in\mathcal{P},\,p>m}\1_{\{\vert B_{p,r}\vert\geq A\}}k\P\left\{\lambda_p(X^{(n)})\geq 1\right\}\\
&=k\sum_{p\in\mathcal{P},\,p>m}\1_{\{\vert B_{p,r}\vert\geq A\}}\frac{1}{n}\left\lfloor\frac{n}{p}\right\rfloor\\
&\leq k\sum_{p\in\mathcal{P},\,p>m}\frac{\1_{\{\vert B_{p,r}\vert\geq A\}}}{p}\to 0,
\end{align*}
as $m\to\infty$, by assumption \ref{it:1} in Proposition~\ref{prop:limit_convergence}.

It remains to check that
\begin{equation}\label{eq:bill3}
\lim_{m\to\infty}\limsup_{n\to\infty}\P\left\{\left\vert\sum_{p\in\mathcal{P},\,p>m}B_{p,r}\1_{\{\vert B_{p,r}\vert \leq A,\sum_{j=1}^{k}\lambda_p(X_j^{(n)})=1\}}\right\vert>\frac{\varepsilon}{2}\right\}=0,
\end{equation}
see \eqref{eq:second_probab_to_estimate}. To that aim, we first notice that
\begin{equation*}
\left\{\sum_{j=1}^{k}\lambda_p(X_j^{(n)})=1\right\}=\bigcup_{j=1}^{k}\left\{\lambda_p(X_j^{(n)})=1,\lambda_p(X_i^{(n)})=0,\forall i\neq j\right\}=:\bigcup_{j=1}^{k}C_{j,p,n}.
\end{equation*}
Moreover, the events $(C_{j,p,n})_{j=1,\ldots,k}$ are disjoint and equiprobable. Thus, the limit \eqref{eq:bill3} follows if we can check that
\begin{equation*}
     \lim_{m\to\infty}\limsup_{n\to\infty}\P\left\{\left\vert\sum_{p\in\mathcal{P},\,p>m}B'_{p,r}\1_{C_{1,p,n}}\right\vert>\frac{\varepsilon}{2k}\right\}=0,
\end{equation*}
where $B'_{p,r}:=B_{p,r}\1_{\{\vert B_{p,r}\vert \leq A\}}$.

Keeping in mind that $\lambda_p(X^{(n)})=0$ for $p>n$, we see that that it is enough to prove that
\begin{equation}\label{eq:bill41}
     \lim_{m\to\infty}\limsup_{n\to\infty}\P\left\{\left\vert\sum_{p\in\mathcal{P},\,p\in(m,\sqrt{n}]}B'_{p,r}\1_{C_{1,p,n}}\right\vert>\frac{\varepsilon}{4k}\right\}=0
\end{equation}
as well as
\begin{equation}\label{eq:bill42}
\lim_{m\to\infty}\limsup_{n\to\infty}\P\left\{\left\vert\sum_{p\in\mathcal{P},\,p\in(\sqrt{n},n]}B'_{p,r}\1_{C_{1,p,n}}\right\vert>\frac{\varepsilon}{4k}\right\}=0.
\end{equation}
Note that $C_{1,p,n}\cap C_{1,q,n}=\varnothing$ if $p,q>\sqrt{n}$, thus \eqref{eq:bill42} is equivalent to
\begin{equation*}
     \lim_{m\to\infty}\limsup_{n\to\infty}\P\left\{\sum_{p\in\mathcal{P},\,p\in(\sqrt{n},n]}(B'_{p,r})^2\1_{C_{1,p,n}}>\frac{\varepsilon^2}{16k^2}\right\}=0.
\end{equation*}
The latter relation follows from Markov's inequality, since
\begin{align*}
&\hspace{-1cm}\P\left\{\sum_{p\in\mathcal{P},\,p\in(\sqrt{n},n]}(B'_{p,r})^2\1_{C_{1,p,n}}>\frac{\varepsilon^2}{16k^2}\right\}\\
&\leq \frac{16k^2}{\varepsilon^2}\sum_{p\in\mathcal{P},\,p\in(\sqrt{n},n]}(B'_{p,r})^2\P\{C_{1,p,n}\}\\
&\leq \frac{16k^2}{\varepsilon^2}\sum_{p\in\mathcal{P},\,p\in(n,\sqrt{n}]}(B'_{p,r})^2\P\{\lambda_p(X_1^{(n)})
\geq 1\}\\
&= \frac{16k^2}{\varepsilon^2} \sum_{p\in\mathcal{P},\,p\in(n,\sqrt{n}]}B^2_{p,r}\1_{\{\vert B_{p,r}\vert \leq A\}}\frac{1}{n}\left\lfloor \frac{n}{p}\right\rfloor\\
&\leq \frac{16k^2}{\varepsilon^2} \sum_{p\in\mathcal{P},\,p\in(n,\sqrt{n}]}\frac{B^2_{p,r}\1_{\{\vert B_{p,r}\vert \leq A\}}}{p}.
\end{align*}
The latter sum converges to zero as $n\to\infty$, by assumption \ref{it:3} in Proposition~\ref{prop:limit_convergence}.

In order to derive \eqref{eq:bill41}, we again use Markov's inequality to obtain
\begin{equation*}
\P\left\{\left\vert\sum_{p\in\mathcal{P},\,p\in(m,\sqrt{n}]}B'_{p,r}\1_{C_{1,p,n}}\right\vert>\frac{\varepsilon}{4k}\right\}
\leq \frac{16k^2}{\varepsilon^2}\E \left\{\left(\sum_{p\in\mathcal{P},\,p\in(m,\sqrt{n}]}B'_{p,r}\1_{C_{1,p,n}}\right)^2\right\},
\end{equation*}
and, further,
\begin{multline*}
\E\left\{\left(\sum_{p\in\mathcal{P},\,p\in(m,\sqrt{n}]}B'_{p,r}\1_{C_{1,p,n}}\right)^2\right\}\\=\sum_{p\in\mathcal{P},\,p\in(m,\sqrt{n}]}(B'_{p,r})^2\P\{C_{1,p,n}\}
+\sum_{p,q\in\mathcal{P},\,p,q\in(m,\sqrt{n}],\,p\neq q}B'_{p,r}B'_{q,r}\P\{C_{1,p,n}\cap C_{1,q,n}\}.
\end{multline*}
We have already estimated the first sum, and thus focus only on the second one. Firstly, as $pq\leq n$ and using part \ref{itt:3} of Lemma~\ref{lem:conv_to_geom}, we may write
\begin{align*}
&\P\{C_{1,p,n}\cap C_{1,q,n}\}\\
&=\P\{\lambda_p(X_1^{(n)})= 1,\lambda_q(X_1^{(n)})=1\} (\P\{\lambda_p(X_1^{(n)})=0,\lambda_q(X_1^{(n)})=0\})^{k-1}\\
&=\left(\left(1-\frac{1}{p}\right)\left(1-\frac{1}{q}\right)\frac{1}{pq}+O\left(\frac{1}{n}\right)\right)
\left(\left(1-\frac{1}{p}\right)\left(1-\frac{1}{q}\right)+O\left(\frac{1}{n}\right)\right)^{k-1}\\
&=\left(1-\frac{1}{p}\right)^k\left(1-\frac{1}{q}\right)^k\frac{1}{pq}+O\left(\frac{1}{n}\right).
\end{align*}
With the above expansion at hand, we have
\begin{align*}
&\hspace{-1cm}\sum_{p,q\in\mathcal{P},\,p,q\in(m,\sqrt{n}],\,p\neq q}B'_{p,r}B'_{q,r}\P\{C_{1,p,n}\cap C_{1,q,n}\}\\
&=\sum_{p,q\in\mathcal{P},\,p,q\in(m,\sqrt{n}],\,p\neq q}\left(\left(1-\frac{1}{p}\right)^k\frac{B'_{p,r}}{p}\right)\left(\left(1-\frac{1}{q}\right)^k\frac{B'_{q,r}}{q}\right)\\
&\quad+O\left(\frac{1}{n}\sum_{p,q\in\mathcal{P},\,p,q\in(m,\sqrt{n}],\,p\neq q}1\right).
\end{align*}
With $\pi(x)$ denoting the number of primes $p\leq x$, we have
\begin{equation*}
     \frac{1}{n}\sum_{p,q\in\mathcal{P},\,p,q\in(m,\sqrt{n}]}1\leq \frac{1}{n}\sum_{p,q\in\mathcal{P},\,p,q\leq\sqrt{n}}1\leq \frac{1}{n}\pi(\sqrt{n})^2\to 0
\end{equation*}
as $n\to\infty$, since $\pi(x)=o(x)$ as $x\to\infty$ by the prime number theorem.

Finally, using assumption \ref{it:2} of Proposition~\ref{prop:limit_convergence}, we see that
\begin{align*}
\lim_{m\to\infty}\limsup_{n\to\infty}&\sum_{p,q\in\mathcal{P},\,p,q\in(m,\sqrt{n}]}\left(\left(1-\frac{1}{p}\right)^k\frac{B'_{p,r}}{p}\right)\left(\left(1-\frac{1}{q}\right)^k\frac{B'_{q,r}}{q}\right)\\
&=\lim_{m\to\infty}\limsup_{n\to\infty}\left(\sum_{p\in\mathcal{P},\,p\in(m,\sqrt{n}]}\left(\left(1-\frac{1}{p}\right)^k\frac{B'_{p,r}}{p}\right)\right)^2\\&=0,
\end{align*}
and \eqref{eq:bill41} follows.

Summarizing, we see that \ref{it-main:1} implies \ref{it-main:3}. To see that \ref{it-main:3} implies \ref{it-main:4}, just note that by Lemma~\ref{lem:conv_to_geom} \ref{itt:2} and \ref{itt:3}, we have
\begin{equation*}
     \left(\frac{f(L_n(k))}{\bigl(X_1^{(n)}X_2^{(n)}\cdots X_k^{(n)}\bigr)^{r}},\frac{X_1^{(n)}}{n},\ldots,\frac{X_k^{(n)}}{n}\right)\todistr  \left(X_{f,\infty,k},U_1,\ldots,U_k\right).
\end{equation*}
Therefore, \eqref{eq:thm1_conv2}  follows by the continuity of multiplication and the continuous mapping theorem. Obviously, \ref{it-main:4} implies \ref{it-main:1}. The proof of Theorem~\ref{thm:main1} is complete.
\end{proof}

\begin{proof}[Proof of Proposition~\ref{prop:z_k_p_prop}]
We start with an auxiliary lemma, which in our opinion is interesting in its own and will be extended in Appendix \ref{app:generalization_diophantine} to more general Diophantine equations.

\begin{lemma}\label{lem:compositions}
Let $k\in\N$ and $\ell,m\in\N_0$ be fixed integers. Then
\begin{equation}
\label{eq:compositions1}
     \sum_{a_1,\ldots,a_k=0}^{\infty}\1_{\{a_1+\cdots+a_k=\ell,\vee_{j=1}^{k}a_j\leq m\}}=[z^\ell]\left(\frac{1-z^{m+1}}{1-z}\right)^k
\end{equation}
and
\begin{equation}
\label{eq:compositions2}
     \sum_{a_1,\ldots,a_k=0}^{\infty}\1_{\{a_1+\cdots+a_k=\ell,\vee_{j=1}^{k}a_j=m\}}=[z^\ell]\left(\left(\frac{1-z^{m+1}}{1-z}\right)^k-\left(\frac{1-z^{m}}{1-z}\right)^k\right).
\end{equation}
\end{lemma}
\begin{proof}
It is known that the number of compositions of $\ell$ into $i\in\N_0$ summands from the set $\{1,2,\ldots,m\}$ is given by
\begin{equation*}
     [z^\ell]\left(z\left(\frac{1-z^m}{1-z}\right)\right)^i,
\end{equation*}
see e.g.\ {\bf I.15} on page~45 in \cite{Flajolet+Sedgewick:2009}. Note that the quantity on the left-hand side of~\eqref{eq:compositions1} is equal to the number of compositions of $\ell$ into $k$ summands from the set $\{0,1,2,\ldots,m\}$. Let $i\in\{0,\ldots,k\}$ be the number of non-zero summands. Then
\begin{align*}
     \sum_{a_1,\ldots,a_k=0}^{\infty}\1_{\{a_1+\cdots+a_k=l,\vee_{j=1}^{k}a_j\leq m\}}&=\sum_{i=0}^{k}\binom{k}{i}[z^\ell]\left(z\left(\frac{1-z^m}{1-z}\right)\right)^i\\
     &=[z^\ell]\sum_{i=0}^{k}\binom{k}{i}\left(z\left(\frac{1-z^m}{1-z}\right)\right)^i\\
     &=[z^\ell]\left(1+z\left(\frac{1-z^m}{1-z}\right)\right)^k\\
     &=[z^\ell]\left(\frac{1-z^{m+1}}{1-z}\right)^k,
\end{align*}
which proves \eqref{eq:compositions1}. Formula \eqref{eq:compositions2} follows by subtraction.
\end{proof}

Now we are in position to prove Proposition~\ref{prop:z_k_p_prop}. We have:
\begin{align*}
\E t^{Z_k(p)}&=\sum_{\ell=0}^{\infty}t^\ell \sum_{a_1,\ldots,a_k=0}^{\infty}\left(1-\frac{1}{p}\right)^k\frac{1}{p^{a_1+\cdots+a_k}}\1_{\{a_1+\cdots+a_k-\vee_{j=1}^{k}a_j=\ell\}}\\
&=\left(1-\frac{1}{p}\right)^k\sum_{\ell=0}^{\infty}t^\ell \sum_{m=0}^{\infty}\sum_{a_1,\ldots,a_k=0}^{\infty}\frac{1}{p^{\ell+m}}\1_{\{a_1+\cdots+a_k=\ell+m,\vee_{j=1}^{k}a_j=m\}}\\
&=\left(1-\frac{1}{p}\right)^k\sum_{\ell=0}^{\infty}\left(\frac{t}{p}\right)^\ell \sum_{m=0}^{\infty}p^{-m}\sum_{a_1,\ldots,a_k=0}^{\infty}\1_{\{a_1+\cdots+a_k=\ell+m,\vee_{j=1}^{k}a_j=m\}}.
\end{align*}
Using Lemma~\ref{lem:compositions} we continue as follows
\begin{align*}
\E t^{Z_k(p)}&=\left(1-\frac{1}{p}\right)^k\sum_{m=0}^{\infty}p^{-m}\sum_{l=0}^{\infty}(t/p)^l [z^{l+m}]\left[\left(\frac{1-z^{m+1}}{1-z}\right)^k-\left(\frac{1-z^{m}}{1-z}\right)^k\right]\\
&=\left(1-\frac{1}{p}\right)^k\sum_{m=0}^{\infty}p^{-m}\sum_{l=0}^{\infty}(t/p)^l [z^{l}]\left[z^{-m}\left(\left(\frac{1-z^{m+1}}{1-z}\right)^k-\left(\frac{1-z^{m}}{1-z}\right)^k\right)\right]\\
&=\left(1-\frac{1}{p}\right)^k\sum_{m=0}^{\infty}t^{-m}\left[\left(\left(\frac{1-(t/p)^{-(m+1)}}{1-t/p}\right)^k-\left(\frac{1-(t/p)^{-m}}{1-t/p}\right)^k\right)\right],
\end{align*}
where the last equality follows by evaluating the term in square brackets at $z=t/p$. The claim of lemma is now a simple consequence of the binomial theorem and subsequent evaluation of resulting geometric series.
\end{proof}

\appendix
\section{On a Diophantine equation}
\label{app:generalization_diophantine}
In passing, the proof of Proposition~\ref{prop:z_k_p_prop} shows that the number $q_l$ of solutions $(a_1,\ldots ,a_k)\in\mathbb N_0^k$ of the Diophantine equation
\begin{equation}
\label{eq:dio0}
     a_1+\cdots +a_k-\vee_{j=1}^{k}a_j=\ell,\quad \ell\in\N,
\end{equation}
has a rational generating function which can be expressed as follows:
\begin{equation*}
     \sum_{\ell=0}^{\infty}q_{\ell}t^{\ell}=\sum_{m=0}^{\infty}t^{-m}\left(\left(\frac{1-t^{m+1}}{1-t}\right)^{k}-\left(\frac{1-t^{m}}{1-t}\right)^{k}\right).
\end{equation*}
This may be generalized in the following way. For fixed $(x_1,\ldots,x_k)\in\mathbb N^k$ and $b\in\mathbb N$, consider the Diophantine equation
\begin{equation}
\label{eq:dio}
     x_1a_1+\cdots +x_ka_k-b\vee_{j=1}^{k}a_j=\ell
\end{equation}
and denote by $q_\ell$ the number of solutions $(a_1,\ldots ,a_k)\in\mathbb N_0^k$ to \eqref{eq:dio}.
\begin{theorem}
We have
\begin{equation*}
     f_{k,b}^{(x_i)}(t):=\sum_{\ell=0}^{\infty}q_\ell t^\ell =\sum_{m=0}^{\infty}t^{-b m}\left(\prod_{j=1}^{k}\frac{1-t^{(m+1)x_j}}{1-t^{x_j}}-\prod_{i=1}^{k}\frac{1-t^{mx_j}}{1-t^{x_j}}\right),\quad |t|<1.
\end{equation*}
In particular, the generating function $f_{k,b}^{(x_i)}$ is rational.
\end{theorem}
\begin{proof}
Decomposing upon the value of the maximum of $a_1,\ldots,a_k$, one may write
\begin{align*}
     f_{k,b}^{(x_i)}(t)&=\sum_{a_1,\ldots,a_k=0}^{\infty} t^{x_1a_1+\cdots +x_ka_k-b\vee_{j=1}^{k}a_j}\\&=\sum_{m=0}^{\infty} t^{-b m}\sum_{a_1,\ldots,a_k=0}^{\infty} t^{x_1a_1+\cdots +x_ka_k}\1_{\{\vee_{j=1}^{k}a_j=m\}}.
\end{align*}
Further,
\begin{equation*}
     \sum_{a_1,\ldots,a_k=0}^{\infty} t^{x_1a_1+\cdots +x_ka_k}\1_{\{\vee_{j=1}^{k}a_j=m\}}=T_{k,m}^{(x_i)}(t)-T_{k,m-1}^{(x_i)}(t),
\end{equation*}
where we have put
\begin{equation*}
     T_{k,m}^{(x_i)}(t) = \sum_{a_1,\ldots,a_k=0}^{\infty} t^{x_1a_1+\cdots +x_ka_k}\1_{\{\vee_{j=1}^{k}a_j\leq m\}}.
\end{equation*}
It remains to apply formula (6) in \cite{Faaland:1972}, which is an extension of \eqref{eq:compositions1}, to obtain
\begin{equation*}
     T_{k,m}^{(x_i)}(t) = \prod_{j=1}^{k} \sum_{s=0}^{m}t^{sa_j}= \prod_{j=1}^{k}\frac{1-t^{(m+1)a_j}}{1-t^{a_j}}.\qedhere
\end{equation*}
\end{proof}

\section*{Acknowledgments}
We would like to express our gratitude to Djalil Chafa\"i and Richard Stanley for interesting discussions and to thank the anonymous referee for drawing our attention to related papers.


\end{document}